\documentclass[reqno,11pt]{article}
\usepackage{hyperref}
\usepackage{a4wide,eucal,enumerate,mathrsfs,xspace}
\usepackage{amsmath,amssymb,epsfig,bbm}
\usepackage[usenames]{color}

\numberwithin{equation}{section}

\newtheorem{theorem}{Theorem}[section]

\newtheorem{corollary}[theorem]{Corollary}
\newtheorem{lemma}[theorem]{Lemma}
\newtheorem{proposition}[theorem]{Proposition}
\newtheorem{definition}[theorem]{Definition}

\newcommand{\C}{\mathbb{C}}
\newcommand{\D}{\mathbb{D}}

\newcommand{\M}{\mathbb{M}}
\newcommand{\N}{\mathbb{N}}

\newcommand{\R}{\mathbb{R}}

\newcommand{\BB}{\mathscr{B}}

\newcommand{\MM}{\mathscr{M}}

\newcommand{\calC}{{\ensuremath{\mathcal C}}}

\newcommand{\cE}{{\ensuremath{\mathcal E}}}

\newcommand{\cG}{{\ensuremath{\mathcal G}}}

\newcommand{\cV}{{\ensuremath{\mathcal V}}}

\newcommand{\ff}{{\mbox{\boldmath$f$}}}

\newcommand{\mm}{{\mbox{\boldmath$m$}}}

\newcommand{\rr}{{\mbox{\boldmath$r$}}}

\newcommand{\ggamma}{{\mbox{\boldmath$\gamma$}}}

\newcommand{\mmu}{{\mbox{\boldmath$\mu$}}}

\newcommand{\ssigma}{{\mbox{\boldmath$\sigma$}}}

\newcommand{\smmu}{{\mbox{\scriptsize\boldmath$\mu$}}}

\newcommand{\sfc}{{\sf c}}
\newcommand{\sfd}{{\sf d}}

\newcommand{\sfg}{{\sf g}}
\newcommand{\sfh}{{\sf h}}

\newcommand{\sfH}{{\sf H}}

\newcommand{\sfP}{{\sf P}}

\newcommand{\rme}{{\mathrm e}}

\newcommand{\rmC}{{\mathrm C}}
\newcommand{\rmD}{{\mathrm D}}

\newcommand{\rmH}{{\mathrm H}}
\newcommand{\rmI}{{\mathrm I}}

\newcommand{\Kliminf}{K\kern-3pt-\kern-2pt\mathop{\rm lim\,inf}\limits}  
\newcommand{\supp}{\mathop{\rm supp}\nolimits}   

\newcommand{\Lip}{\mathop{\rm Lip}\nolimits}          
\renewcommand{\d}{{\mathrm d}}

\newcommand{\restr}[1]{\lower3pt\hbox{$|_{#1}$}}
\newcommand{\topref}[2]{\stackrel{\eqref{#1}}#2}
\newcommand{\la}{{\langle}}                  

\newcommand{\down}{\downarrow}              
\newcommand{\up}{\uparrow}
\newcommand{\weakto}{\rightharpoonup}

\newcommand{\eps}{\varepsilon}  
\newcommand{\nchi}{{\raise.3ex\hbox{$\chi$}}}

\newcommand{\forevery}{\text{for every }}

\def\qed{\ifmmode 
  \else \leavevmode\unskip\penalty9999 \hbox{}\nobreak\hfill
  \fi               
    \qquad           \hbox{\hskip.5em $\square$
                \hskip.1em}}
\def\endproofsym{\qed}

\newenvironment{proof}[1][Proof]{\def\endproofsym{\qed}\trivlist\item[\hskip\labelsep{%
\noindent{\normalfont\emph{#1}.}\hskip .321429\parindent}]\ignorespaces}
{\endproofsym\endtrivlist}

\newcommand{\ent}[1]{\mathrm{Ent}_{\mm}(#1)}              

\newcommand{\entv}{\mathrm{Ent}_{\mm}}                    

\renewcommand{\mm}{\mathfrak m}

\newcommand{\BorelSets}[1]{\BB(#1)}

\newcommand{\Probabilities}[1]{\mathscr P(#1)}          

\newcommand{\ProbabilitiesTwo}[1]{\mathscr P_2(#1)}     

\newcommand{\AdmissiblePlanII}[2]{\Gamma(#1,#2)}        

\renewcommand{\C}{\mathsf{Ch}}

\newcommand{\LHeat}[1]{\sfP_{\kern-2pt#1}}

\newcommand{\RCD}[2]{\mathrm{RCD}(#1,#2)}

\newcommand{\CD}[2]{CD(#1,#2)}

\newcommand{\BE}[2]{BE(#1,#2)}   

\newcommand{\Gbil}[2]{\Gamma\big(#1,#2\big)}   
\newcommand{\Gq}[1]{\Gamma\big(#1\big)}
\newcommand{\Gdq}[1]{\Gamma_2(#1)}
\newcommand{\gdq}[1]{\gamma_2(#1)}

\newcommand{\Bilinear}[2]{\mathcal E(#1,#2)}
\newcommand{\Dirichlet}[1]{\mathcal E(#1)}

\newcommand{\Edmeas}[1]{\Gamma^\star_2(#1)}
\newcommand{\Esmeas}[1]{\Gamma^\perp_2(#1)}
\newcommand{\edmeas}[1]{\gamma_2(#1)}

\newcommand{\Ebilmeas}[2]{\Gamma^\star_2(#1,#2)}
\newcommand{\Ebilsmeas}[2]{\Gamma^\perp_2(#1,#2)}
\newcommand{\ebilmeas}[2]{\gamma_2(#1,#2)}

\newcommand{\DeltaE}{\Delta_\cE}

\renewcommand{\C}{{\sf Ch}}

\newcommand{\Capacity}{\mathrm{Cap}}
\newcommand{\Cp}{\Capacity}

\newcommand{\heat}[2]{\mathsf P_{#1}#2}

\newcommand{\DeltaEM}{{\DeltaE^{\star}}}
\renewcommand{\BE}[2]{\mathrm{BE}(#1,#2)}
\renewcommand{\CD}[2]{\mathrm{CD}(#1,#2)}\renewcommand{\RCD}[2]{\mathrm{RCD}(#1,#2)}
\renewcommand{\cV}{\mathbb {V}}
\renewcommand{\cG}{\mathbb{G}}
\newcommand{\cGz}{\cG_{\infty}}
\newcommand{\cVz}{\cV_{\kern-2pt\infty}}
\newcommand{\Gqd}[1]{\Gamma_2(#1)}
\newcommand{\Gqlin}[2]{\mathbf{\Gamma}_2[#1;#2]}
\newcommand{\cVu}{\cV^1_{\kern-2pt\infty}}
\newcommand{\cVd}{\cV^2_{\kern-2pt\infty}}

\newcommand{\DG}{\D_\infty}

\usepackage[normalem]{ulem}
\usepackage{pdfsync}

\newcommand{\lref}[1]{$\langle$L.\ref{#1}$\rangle$}

\newcommand{\cLz}{\mathbb G_\infty}

\newcommand{\cDv}{\DG}
\newcommand{\cMz}{\mathbb M_\infty}
\newcommand{\tril}[3]{\rmH [#1](#2,#3)}

\newcommand{\QCh}{\textbf{\upshape[Q-Ch]}}
\newcommand{\CDKI}{$\mbox{\boldmath{$[\CD K\infty]$}}$}
\newcommand{\Length}{\textbf{\upshape[Length]}}
\newcommand{\Cont}{\textbf{\upshape[Cont]}}
\newcommand{\Wcont}{$\mbox{\boldmath{$[W_2\text{\textbf{\upshape-cont}}]$}}$}
\newcommand{\BEKI}{$\mbox{\boldmath{$[\BE K\infty]$}}$}
\newcommand{\EVIK}{$\mbox{\boldmath{$[\mathrm{EVI}_K]$}}$}

\renewcommand{\DeltaE}{\mathsf L}
\renewcommand{\AdmissiblePlanII}[2]{\Pi(#1,#2)}

\title{Self-improvement of the Bakry-\'Emery condition and
  Wasserstein contraction of the heat flow in \\
  $\RCD K\infty$ metric measure spaces}
\begin{document}

\author{
   Giuseppe Savar\'e\
   \thanks{Universit\`a di Pavia. email:
   \textsf{giuseppe.savare@unipv.it}. Partially supported by
   PRIN10-11 grant from MIUR for the project \emph{Calculus of Variations}.}
   }

\maketitle

\begin{abstract} 
  We prove that the linear ``heat'' flow in a $\RCD K\infty$ metric
  measure space $(X,\sfd,\mm)$ satisfies a contraction property
  with respect to every $L^p$-Kantorovich-Rubinstein-Wasser\-stein
  distance,
  $p\in [1,\infty]$. In particular, we obtain a precise estimate for the optimal
  $W_\infty$-coupling between two fundamental solutions 
  in terms of the distance of the initial points.

  The result is a consequence of the equivalence between 
  the $\RCD K\infty$ lower Ricci bound and the 
  corresponding Bakry-\'Emery condition 
  for the canonical Cheeger-Dirichlet form in $(X,\sfd,\mm)$.
  The crucial tool is the extension to the non-smooth metric measure
  setting of the Bakry's argument, that allows to improve
  the commutation estimates between the Markov semigroup and
  the \emph{Carr\'e du Champ} $\Gamma$ associated to the Dirichlet form.

  This extension is based on a new a priori estimate and a capacitary
  argument for regular and tight Dirichlet forms that are of
  independent interest.
\end{abstract}

\tableofcontents


\section{Introduction}

The investigation of the deep connections between lower Ricci
curvature bounds
(also in the broader sense of the Bakry-\'Emery curvature-dimension
condition $\BE KN$ \cite{Bakry-Emery84})
and optimal transport in Riemannian geometry started with the 
pioneering papers
\cite{Otto-Villani00,Codero-McCann-Schmuckenschlager01}. 
Since then a big effort have been made
to develop a synthetic theory
of curvature-dimension bounds for a general
metric-measure space $(X,\sfd,\mm)$ in absence
of a smooth differential structure.  

\subsubsection*{Lott-Sturm-Villani $\CD K\infty$ spaces}

In the approach developed by Sturm \cite{Sturm06I,Sturm06II} and 
Lott-Villani \cite{Lott-Villani09}
(see also \cite{Villani09}),
optimal transport provides a very useful and 
far-reaching point of view, in particular 
to obtain a stable notion with respect to measured 
Gromov-Hausdorff (or Gromov-Prokhorov) convergence
that includes all possibile Gromov-Hausdorff limits
of Riemannian manifolds under uniform dimension and lower curvature bounds
\cite{Cheeger-Colding97I,Cheeger-Colding00II,Cheeger-Colding00III}.

According to Lott-Sturm-Villani, a complete and separable 
metric space $(X,\sfd)$ endowed with a Borel probability
measure $\mm\in \Probabilities X$ (here we assume $\mm(X)=1$ for
simplicity, see \S\,\ref{subsec:basic} for a more general condition)
satisfies the $\CD K\infty$ curvature bound
if
the relative entropy functional $\entv:\Probabilities X\to
[0,\infty]$ induced by $\mm$ is
displacement $K$-convex in the Wasserstein space
$\ProbabilitiesTwo X$. The latter is the 
space of Borel probability measures with
finite quadratic moment endowed with the $L^2$
Kantorovich-Rubinstein-Wasserstein distance $W_2$,
see \S\,\ref{subsec:basic}.

A question that naturally arises in this metric setting concerns the 
relationships between the optimal transport and
the Bakry-\'Emery's approaches. 
Since the latter makes sense only in the framework of a Dirichlet form $\cE$
generating a linear Markov semigroup $(\heat t)_{t\ge0}$ in $L^2(X,\mm)$,
one has first to understand how to construct
a diffusion semigroup and 
an energy functional in a $\CD K\infty$ space.

Since the $\CD K\infty$ condition involves the geodesic
$K$-convexity of the entropy functional in 
the Wasserstein space, it is quite natural to consider 
the metric gradient flow $(\sfH_t)_{t\ge0}$ \cite{AGS08} 
of $\entv$ in $(\ProbabilitiesTwo X,W_2)$ 
(see \cite{Gigli10,AGS11a}). 
As showed initially by 
\cite{Jordan-Kinderlehrer-Otto98} in $\R^n$
and then extended to many different situations
by
\cite{Erbar10,Villani09,Ambrosio-Savare-Zambotti09,Sturm-Ohta-CPAM,GigliKuwadaOhta10},
it turns out \cite{AGS11b} that $(\sfH_t)_{t\ge0}$ essentially coincides with the 
$L^2$-gradient flow $(\sfP_t)_{t\ge0}$ of the 
convex and lower semicontinuous \emph{Cheeger energy}
\begin{equation}
  \label{eq:37pre}
  \C (f):=\inf\Big\{\liminf_{n\to\infty}\frac12\int_X |\rmD
  f_n|^2\,\d\mm:f_n\in \Lip_b(X),\quad
  f_n\to f\text{ in }L^2(X,\mm)\Big\},
\end{equation}
where the metric slope $|\rmD f|$ of a Lipschitz function $f:X\to \R$ 
is defined by
$|\rmD f|(x):=\limsup_{y\to x}|f(y)-f(x)|/\sfd(x,y)$.

$(\sfP_t)_{t\ge0}$ thus defines a (possibly nonlinear) semigroup of contractions in
$L^2(X,\mm)$ and, in fact, in any $L^p(X,\mm)$. Since it is also
positivity
preserving, it is a Markov semigroup if and only if
it is linear, or, equivalently, if $\C$ is a quadratic form in
$L^2(X,\mm)$,
thus satisfying the parallelogram rule
  \begin{equation}
  \label{eq:38pre}
  \C(f+g)+\C(f-g)=2\C(f)+2\C(g)\quad\text{for every }f,g\in D(\C).
  \tag{Q-Ch}
\end{equation}

\subsubsection*{$\RCD K\infty$-metric measure spaces and the 
Bakry-\'Emery $\BE K\infty$ condition}
Spaces satisfying Lott-Sturm-Villani $\CD K\infty$ conditions and 
\eqref{eq:38pre} have been introduced in 
\cite{AGS11b} as metric measure spaces with \emph{Riemannian Ricci
  curvature} bounded from below, $\RCD K\infty$ spaces in short.
This more restrictive class of spaces can also be characterised in terms of the 
\emph{Evolution variational inequality} formulation of
$(\sfH_t)_{t\ge0}$, see \eqref{eq:65}, 
that provides the $W_2$ contraction property
\begin{equation}
  \label{eq:56}
  W_2(\sfH_t \mu,\sfH_t \nu)\le \rme^{-K t}W_2(\mu,\nu)\quad
  \forevery \mu,\nu\in \ProbabilitiesTwo X.
\end{equation}
The $\RCD K\infty$ condition is still stable 
with respect to measured Gromov-Hausdorff convergence
\cite{AGS11b,GMS13} and thus
includes all possibile measured Gromov-Hausdorff limits
of Riemannian manifolds under uniform lower curvature bounds.

In $\RCD K\infty$ spaces 
$\cE:=2\C$ is a strongly local Dirichlet form
admitting a \emph{Carr\'e du champ} $\Gamma(f)$
that coincides with the squared minimal weak upper gradient 
$|\rmD f|_w^2$ associated to \eqref{eq:37pre}, see \eqref{eq:84} and \eqref{eq:51}.
In terms of the generator $\DeltaE:D(\DeltaE)\subset L^2(X,\mm)\to
L^2(X,\mm)$ 
of $(\sfP_t)_{t\ge0}$ this provides useful  
the Leibnitz and composition rules
\begin{displaymath}
  2\Gbil fg=\DeltaE(fg)-f\DeltaE g-g\DeltaE f,\quad
  \DeltaE(\Phi(f))=\Phi'(f)\DeltaE f+\Phi''(f)\Gq f
\end{displaymath}
at least for a suitable class of functions in $D(\DeltaE)$, see 
\S\,\ref{subsec:Leibnitz}. 

Distance and energy are intimately 
correlated by the explicit formula 
\eqref{eq:37pre} (that involves the metric slope of Lipschitz
functions) and by the somehow dual property that
expresses $\sfd$ as the canonical distance 
\cite{Biroli-Mosco95} associated to $\cE$:
\begin{subequations}
  \begin{gather}
    \label{eq:3}
    \text{every bounded function $f\in D(\cE)$ with $\Gq f\le 1$ has a
      continuous representative $\tilde f$,}\\
    \label{eq:78}
    \sfd(x,y):=\sup\Big\{\psi(x)-\psi(y):\psi\in D(\cE)\cap
    \rmC_b(X),\quad \Gq f\le 1\Big\}.
  \end{gather}
\end{subequations}
Having a Carr\'e du champ at disposal, it is
then possibile to consider a weak version
(see \eqref{eq:22})  of the
\emph{Carr\'e du champ it\'er\'e} 
\begin{equation}
  \label{eq:39}
  2\Gamma_2(f,g):=\DeltaE\Gbil fg-\Gbil f{\DeltaE g}-
  \Gbil g{\DeltaE f},
\end{equation}
and to prove a weak $\BE K\infty$ condition of the type
\begin{equation}
  \label{eq:80}
  \Gamma_2(f)\ge K\Gq f,\quad
  \text{where}\quad
    \Gq f:=\Gbil ff,\quad
  \Gamma_2(f):=\Gamma_2(f,f),
\end{equation}
in a suitable weaker  integral form (Definition \ref{def:BE}), but still sufficient to
get
the crucial pointwise gradient bound
\begin{equation}
  \label{eq:55}
  \Gq{\heat t f}\le |\rmD \heat t f|^2\le \rme^{-2K t}\heat t \Gq
  f\quad \forevery f\in \Lip_b(X).
\end{equation}
It turns out that the implication $\RCD K\infty\ \Rightarrow\ \BE
K\infty$ 
can also be inverted and the two points of view are eventually
equivalent.
This has been shown by \cite{AGS12}: starting from
a Polish topological space $(X,\tau)$ endowed with
a local Dirichlet form $\cE$ with the associated \emph{Carr\'e du champ}
$\Gamma$
and the intrinsic distance $\sfd$ satisfying (\ref{eq:3},b) and inducing
the topology $\tau$, if 
$\BE K\infty$ holds, then $(X,\sfd,\mm)$ 
is a $\RCD K\infty$ metric measure space.

\subsubsection*{Applications of $\BE K\infty$: 
refined gradient estimates and Wasserstein contraction}

The identification between $\RCD K\infty$ and $\BE K\infty$ 
lead to the possibility 
to apply a large numbers of the results and techniques
originally proved for smoother spaces satisfying the 
Bakry-\'Emery condition. 
Performing this project is not always simple, since 
proofs often use extra regularity or algebraic assumptions 
(see e.g.\ \cite[Page 24]{Bakry92}) that
prevent a direct application to the non smooth context.

Among the most useful properties, Bakry \cite{Bakry85,Bakry92} 
showed that the $\Gamma_2$ condition expressed through
the pointwise bounds \eqref{eq:55} is potentially self-improving,
since it leads to the stronger commutation inequality
\begin{equation}
  \label{eq:79}
  \Big(\Gq{\heat t f}\Big)^{\alpha}\le \rme^{-2\alpha K t}\heat t\Big( \Gq
  f^\alpha\Big)\quad
  \forevery \alpha\in [1/2,2].
\end{equation}
\eqref{eq:79} is in fact a consequence of the crucial estimate
\begin{equation}
  \label{eq:40}
  \Gq{\Gq f}\le 4\Big(\Gamma_2(f)-K\Gq f\Big)\Gq f,
\end{equation}
a formula whose meaning can be better 
 understood recalling that 
in a Riemannian manifold $(\M^d,\sfg)$ 
endowed with the canonical Riemannian volume $\mm=\mathrm{Vol}_\sfg$,
we have
\begin{equation}
  \label{eq:82}
  \Gq f=|\rmD f|_\sfg^2,\quad
  \Gqd f-K\Gq f\ge|\rmD^2 f|_\sfg^2,\quad
  \Gq{\Gq f}=\Big|\rmD \big|\rmD f\big|_\sfg^2\Big|_\sfg^2\le 
  4|\rmD^2 f|_\sfg^2\,|\rmD f|_\sfg^2.
\end{equation}
\eqref{eq:40}
can be derived by applying 
the $\Gamma_2$ inequality \eqref{eq:80} to 
polynomials of two or more functions $f_1,f_2,\cdots$.
However the Bakry's clever strategy of \cite{Bakry85,Bakry92}  
requires a multivariate differential formula for
the $\Gamma_2$ operator, that typically involves
further smoothness assumptions.

The aim of the present paper is twofold: from one side, we want to
show how to obtain the estimate \eqref{eq:40} in 
a very general setting, starting from the 
weak integral formulation of $\BE K\infty$.

This result is independent of the theory of
metric measure spaces, and
it is obtained for general Dirichlet forms in Polish spaces
satisfying standard regularity and tightness assumptions.
It relies on a simple estimate showing that $\Gq f\in D(\cE)$ 
if $f$ belongs to the space $\DG$, whose elements $f$ are
characterised by
$f\in D(\DeltaE)$ with $\Gq f\in L^\infty(X,\mm),\ 
\DeltaE f\in D(\cE)$. Tightness and regularity of $\cE$ 
are then sufficient to give a measure-theoretic sense to
$\DeltaE\Gq f$, to $\Gdq f$ and to multivariate calculus for
$\Phi\circ f$ thanks to capacitary arguments.
The main point here is that $\Gamma_2(f)$
may be singular with respect to $\mm$, but its singular part
is nonnegative; moreover, the multiplication of
the measure $\Gamma_2(f)$ with functions in $D(\cE)$ 
still makes sense since the latter admit a quasi continuous
representative
and polar sets are negligible w.r.t.\ the measure $\Gamma_2(f)$.

The derivation of \eqref{eq:79} from \eqref{eq:4} follows
then the ideas of \cite{Bakry-Ledoux06,Bakry06,Wang11},
suitably adapted to the weak integral version of \eqref{eq:80}.

Finally, the application of 
\eqref{eq:80} to contraction estimates for
the heat flow $(\sfH_t)_{t\ge0}$ in Wasserstein spaces
follows the Kuwada's duality approach \cite{Kuwada10},
thanks to \eqref{eq:3}, \eqref{eq:78} and 
the refined argument developed in \cite{AGS12}.
We can then prove the optimal contraction estimate for every $L^p$-Wasserstein 
distance
\begin{equation}
  \label{eq:56bis}
  W_p(\sfH_t \mu,\sfH_t \nu)\le \rme^{-K t}W_p(\mu,\nu)\quad
  \forevery \mu,\nu\in \Probabilities X,\quad p\in [1,\infty],
\end{equation}
and, when $K\ge0$, for 
any transport cost depending on the distance $\sfd$ in 
an increasing way, see \eqref{eq:21}
(see \cite{Natile-Peletier-Savare11} for a similar estimate
in $\R^n$). 

\subsubsection*{Plan of the paper}
We will recall in Section \ref{sec:prel} 
a few basic results concerning Dirichlet forms, \emph{Carr\'e du champ}, 
multivariate differential calculus and capacities.
A simple but important estimate is proved in Lemma \ref{le:Lapmeas}.

After a brief review of the weak formulation of
the $\BE K\infty$ condition, Section \ref{sec:BE} 
contains the main properties for the measure theoretic interpretation
of the 
\emph{Carr\'e du champ it\'er\'e} $\Gamma_2$ and
the corresponding multivariate calculus rules.
The main estimates are then proved in Theorem \ref{thm:main1} and its
Corollary
\ref{cor:kuwada}.

Applications to $\RCD K\infty$ spaces 
and to Wasserstein contraction of the heat flow are eventually discussed in the
last
section \ref{sec:RCD}.

\subsubsection*{Acknowledgment}

We would like to thank Luigi Ambrosio, Nicola Gigli, Michel Ledoux
for various fruitful discussions and the anonymous reviewer for the 
accurate and valuable report.

\section{Preliminaries}
\label{sec:prel}

\subsection{Notation, Dirichlet forms and Carr\'e du Champ}
Let $(X,\tau)$ be a Polish topological space.
We will denote by $\BorelSets X$ the collection of its Borel sets
and by $\MM(X)$ the space of 
Borel signed measures with finite total variation, i.e.~$\sigma$-additive maps
$\mu:\BorelSets X\to \R.$
$\MM(X)$ is endowed with the weak convergence with respect to
the duality with the continuous and bounded functions
of $\rmC_b(X)$. $\MM_+(X)$ and $\Probabilities X$ will
denote the convex subsets of 
nonnegative finite measures and of probabilities measures in $X$, 
respectively.

We will consider a $\sigma$-finite Borel measure $\mm\in \MM_+(X)$
with full support $\supp(\mm)=X$ and a strongly local, symmetric Dirichlet form
$\cE:L^2(X,\mm)\to[0,\infty]$ 
with proper domain $\cV:=\big\{f\in L^2(X,\mm):\cE(f)<\infty\big\}$ dense in $L^2(X,\mm)$.
$\cE$ generates a mass preserving Markov semigroup $(\heat t)_{t\ge0}$
in $L^2(X,\mm)$
with generator $\DeltaE$ and
domain $D(\DeltaE)$ dense in $\cV$.

We will still use the symbol $\cE$ to denote the associated bilinear form
in $\cV$.
$\cV$ is an Hilbert space with the graph norm induced by $\cE$:
\begin{displaymath}
  \|f\|^2_\cV:=\|f\|_{L^2(X,\mm)}^2+\cE(f,f).
\end{displaymath}
We will assume that $\cE$ admits a \emph{Carr\'e du Champ} $\Gbil\cdot\cdot$:
it is a symmetric, bilinear and continuous map
$\Gamma:\cV\times\cV\to L^1(X,\mm)$,
which is uniquely characterised
in the algebra $\cV\cap L^\infty(X,\mm)$ by
\begin{displaymath}
  2\int_X \Gbil fg\varphi\,\d\mm=\cE(f,g\varphi)+\cE(g,f\varphi)
  -\cE(fg,\varphi)
  \quad\forevery \varphi\in \cV\cap L^\infty(X,\mm).
\end{displaymath}
In the following we set
\begin{equation}
  \label{eq:20}
  \begin{gathered}
    \cVz:=\cV\cap L^\infty(X,\mm),\quad \cLz:=\{f\in \cVz:\Gq f\in
    L^\infty(X,\mm)\}.
  \end{gathered}
\end{equation}

\subsection{Leibnitz rule and multivariate calculus}
\label{subsec:Leibnitz}
We recall now a few useful calculus rules.
We will consider smooth functions $\Phi,\Psi:\R^n\to \R$ with
$\Phi(0)=\Psi(0)=0$, 
we set $\Phi_i:=\partial_i\Phi$, $\Phi_{ij}:=\partial_{ij}\Phi$,
$i,j=1,\cdots, n$, and similarly for $\Psi$.
We will denote by $\ff:=(f_i)_{i=1}^n$ a $n$-uple of 
real measurable functions defined on $X$ and by $\Phi(\ff)=
\Phi(f_1,\cdots,f_n)$ the corresponding composed function.

For a proof of the following properties, we refer to 
\cite[Ch.~I, \S 6]{Bouleau-Hirsch91}: notice that 
we do not assume any bounds on the derivatives of $\Phi$ and $\Psi$
since they will be composed with (essentially) bounded functions.
\begin{enumerate}[$\langle$L.1$\rangle$] 
\item \label{l1}
  $\cVz$ and $\cLz$ are closed with respect to pointwise
  multiplication
  (see \cite[Ch.~I, Cor.~3.3.2]{Bouleau-Hirsch91} and the next Leibnitz rule \eqref{eq:18}).
\item \label{l2}
  If $f\in \cV$ and $g\in \cLz$ then
  $fg\in \cV$.
\item \label{l3}
  If $f,g\in \cVz$ (or $f\in \cV$ and $g\in \cLz$) and $h\in \cV$ then
  \cite[Ch.~I, Cor.~6.1.3]{Bouleau-Hirsch91}
  \begin{equation}
    \label{eq:18}
    \Gbil{fg}h=f\Gbil gh+g\Gbil fh,\quad
    \Gq {fg}=f^2\Gq g+g^2\Gq f+2fg\Gbil fg.
  \end{equation}
\item \label{l4}
  If $(f_i)_{i=1}^n\in (\cVz)^n$
  the functions $\Phi(\ff)$, $\Psi(\ff)$ belong to $\cVz$ and
  \cite[Ch.~I, Cor.~6.1.3]{Bouleau-Hirsch91} 
  \begin{equation}
    \label{eq:49}
    \Gbil{\Phi(\ff)}{\Psi(\ff)}=\sum_{i,j} \Phi_i(\ff)\Psi_j(\ff)\Gbil{f_i}{f_j}.
  \end{equation}
\item \label{l5}
  If $f_i\in D(\DeltaE)\cap \cLz$ then $\Phi(\ff)\in
  D(\DeltaE)\cap\cLz$ with
  \cite[Ch.~I, Cor.~6.1.4]{Bouleau-Hirsch91} 
\begin{equation}
  \label{eq:15}
  \DeltaE (\Phi(\ff))=
  \sum_i \Phi_i(\ff)\DeltaE f_i+
  \sum_{i,j} \Phi_{ij}(\ff)\Gbil{f_i}{f_j}.
\end{equation}
\item \label{l6}
  $D(\DeltaE)\cap\cLz$ is closed with respect to pointwise
  multiplication: 
  if $f_i\in D(\DeltaE)\cap \cLz$ then
  \begin{equation}
    \label{eq:19}
    \DeltaE (f_1\,f_2)=f_1\,\DeltaE f_2+f_2\,\DeltaE f_1+2\,\Gbil {f_1}{f_2}.
  \end{equation}
\end{enumerate}

\subsection{Quasi-regular Dirichlet forms, capacity and measures with finite energy.}
\label{subsec:cap}
\newcommand{\cVrestr}[1]{\cV_{\kern-2pt #1}}
\newcommand{\cVl}[1]{\cV_{\kern-2pt #1}}
\newcommand{\cp}{\mathrm{Cap}}

We follow here the approach developed by Ma and R\"ockner, 
see 
\cite[III.2, III.3, IV.3]{Ma-Rockner92} 
(covering the general case of a
possibly non-symmetric Dirichlet form) and
\cite[1.3]{Chen-Fukushima12}.
If $F$ is a closed subset of $X$ we set 
\begin{displaymath}
  \cVrestr F:=\Big\{f\in \cV:f(x)=0\text{ for $\mm$-a.e.\ $x\in
    X\setminus F$}\Big\}.
\end{displaymath}
\begin{definition}[Nests, polar sets, and 
  quasi continuity {\cite[III.2.1]{Ma-Rockner92},
    \cite[1.2.12]{Chen-Fukushima12}}]\ \\
  An $\cE$-nest is an increasing sequence $(F_k)_{k\in \N}$ of closed
  subsets of $X$ 
  such that $\cup_{k\in \N}\cVrestr{F_k}$ is dense in
  $\cV$.\\
  A set $N\subset X$ is $\cE$-polar if there is an $\cE$-nest 
  $(F_k)_{k\in\N}$ such that $N\subset X\setminus \cup_{k\in \N} F_k$.
  If a property of points in $X$ 
  holds in a complement of an $\cE$-polar set we say that 
  it holds $\cE$-quasi-everywhere ($\cE$-q.e.). \\
  A function $f:X\to \R$ is said to be $\cE$-quasi-continuous if 
  there exists an $\cE$-nest $(F_k)_{k\in\N}$ such that every restriction
  $f\restr{F_k}$ is continuous on $F_k$. 
\end{definition}
$\cE$-nests and $\cE$-polar sets can also be 
characterized in terms of capacities; we recall here a version
that we will be useful later on. 
The capacity $\cp$ (it corresponds to
$\cp_{h,1}$ with $h\equiv 1$ in the notation of \cite{Chen-Fukushima12}) 
of an open set $A\subset X$
is defined by
$$
\cp(A):=\inf \big\{ \|u\|^2_\cV:\ \text{$u\geq
1$ $\mm$-a.e. in $A$}\big\},
$$%
and it can be extended to arbitrary sets $B\subset X$ by
\begin{displaymath}
  \cp(B):=\inf\big\{\cp(A):B\subset A, A\text{ open}\big\}.
\end{displaymath}
Notice also that  $\cp(A)\geq \mm(A)$. 

\begin{theorem}[{\cite[1.2.14]{Chen-Fukushima12}}]
  \label{thm:criterium}
  Let us suppose that there exists a nondecreasing sequence $(X_n)_{n\in
    \N}$ of 
  open subsets of $X$ such that 
  \begin{equation}
    \label{eq:86}
    \Cp(X_n)<\infty,\quad
    \overline X_n\subset X_{n+1},\quad
    (\overline X_n)_{n\in \N}\text{ is an $\cE$-nest}.
  \end{equation}
  \begin{enumerate}[$(i)$]
  \item A nondecreasing sequence of closed subsets $F_k\subset X$ is an
    $\cE$-nest if and only if
    \\$\lim_{k\to\infty}\cp(X_n\setminus
    F_k)=0$ for every $n\in \N$.
    \item $N\subset X$ is an $\cE$-polar set if and only if
      $\cp(N)=0$.
    \end{enumerate}
\end{theorem}
When $\mm(X)<\infty$ then $\cp(X)=\mm(X)< \infty$, so that
\eqref{eq:86} is always satisfied by choosing $X_n\equiv X$.
In this case a function $f:X\to\R$ is
      $\cE$-quasi-continuous if 
for every $\eps>0$ there exists a closed set $C_\eps\subset X$
such that $f\restr{C_\eps}$ is continuous and 
$\cp(X\setminus C_\eps)<\eps$.
\begin{definition}[Quasi-regular Dirichlet forms]
  \label{def:quasi-regular}
  The Dirichlet form $\cE$ is quasi-regular if 
  \begin{enumerate}[\rm $\langle$QR.1$\rangle$]
  \item 
    \label{item:QR1} There exists an $\cE$-nest $(F_k)_{k\in \N}$ consisting of
    compact sets.
  \item \label{item:QR2}
    There exists a dense subset of $\cV$ whose elements have
    $\cE$-quasi-continuous 
    representatives.
  \item 
    \label{item:QR3} There exists an $\cE$-polar set $N\subset X$ and 
    a countable collection of $\cE$-quasi-continuous
    functions $(f_k)_{k\in \N}\subset \cV$ separating the points
    of $X\setminus N$.
  \end{enumerate}
\end{definition}
If $\cE$ is quasi-regular, then 
\cite[Remark 1.3.9(ii)]{Chen-Fukushima12}
\begin{equation}
\text{every function $f\in \cV$ admits an $\cE$-quasi-continuous
representative $\tilde f$,}
\label{eq:87}
\end{equation}
$\tilde f$ is unique up to q.e.~equality.
Notice that 
\begin{equation}
  \text{if $f\in \cVz$ with $|f|\le M$ $\mm$~a.e.\ in
$X$, then $|\tilde f|\le M$ q.e.}\label{eq:60}
\end{equation}
 When $\mm(X)<\infty$ so that $\cp(X)<\infty$, 
Theorem \ref{thm:criterium}$(i)$ shows that
$\langle$QR.\ref{item:QR1}$\rangle$ is equivalent to the tightness
condition
\begin{equation}
\text{there exist compact sets $K_n\subset X$, $n\geq 1$,
  such that $\lim_{n\to\infty}\cp(X\setminus K_n)= 0$.}\label{eq:tight}
\end{equation}
 In the general case of a $\sigma$-finite measure $\mm$ satisfying
\eqref{eq:86}, we have the
following simple criterium of quasi-regularity, where 
(with a slight abuse of notation) we will denote by 
$\cV\cap \rmC(X)$ the subspace of $\cV$ consisting of those
functions which 
admits a continuous representative.
\begin{lemma}[A criterium for quasi-regularity]
  \label{le:qr-criterium}
  Let us assume that there exists a nondecreasing 
  sequence $(X_n)_{n\in \N}$ of open subsets of
  $X$ satisfying \eqref{eq:86} 
  and let us suppose that 
  \begin{enumerate}[$\langle\mathrm{QR}.1'\rangle$]
  \item For every $n,m\in \N$ there exists a compact set
    $K_{n,m}\subset X$ such that $\cp(X_n\setminus K_{n,m})\le
    1/m$.
    \item
      $\cV\cap \rmC(X)$ is dense in $\cV$ and
      it separates the points of $X$.
  \end{enumerate}
  Then $\cE$ is quasi-regular.
\end{lemma}
\begin{proof}
  Let us set $F_k:=\bigcup_{j=1}^k K_{j,j}$. $(F_k)_{k\in \N}$ is a
  nondecreasing sequence of compact sets and 
  whenever $k\ge n$ we get
  \begin{displaymath}
    \cp(X_n\setminus F_k)\le 
    \cp(X_k\setminus F_k)\le 
    \cp(X_k\setminus K_{k,k})\le 1/k,
  \end{displaymath}
  so that $\lim_{k\to\infty}\cp(X_n\setminus F_k)=0$. Applying 
  Theorem \ref{thm:criterium}$(i)$ we obtain that $(F_k)_{k\in \N}$
  is an $\cE$-nest, so that $\langle\mathrm{QR}.\ref{item:QR1}\rangle$
  holds.

  $\langle\mathrm{QR}.\ref{item:QR2}\rangle$ is a trivial consequence
  of $\langle\mathrm{QR}.2'\rangle$; 
  $\langle\mathrm{QR}.\ref{item:QR3}\rangle$ still follows by 
  $\langle\mathrm{QR}.2'\rangle$ thanks to 
  \cite[Ch.~II, Prop.~4]{Schwartz73}.
\end{proof}
%
%
We introduce the convex set 
$\cVl + :=\big\{\phi\in \cV:\phi\geq
0\ \mm\text{-a.e. in }X\}$; $\cVl +'$ denotes the 
set of linear functionals $\ell\in \cV'$ such that
$\langle \ell,\phi\rangle\geq 0$ for all $\phi\in \cVl +$;
we also set $\cVl \pm':=\cVl +'-\cVl +'$.

By Lax-Milgram Lemma, for every $\ell\in \cVl+'$ there exists a unique
$u_\ell\in \cV$ representing $\ell$ in the sense that 
\begin{equation}
  \label{eq:88}
  \langle \ell,\varphi\rangle=\int_X
  u_\ell\varphi\,\d\mm+\cE(u_\ell,\varphi)
  \quad\forevery\varphi\in \cV.
\end{equation}
$u_\ell$ is $1$-excessive according to 
\cite[Def.~1.2.1, Lemma 1.2.4]{Chen-Fukushima12}
(in particular 
$u_\ell$ is non negative).
The proof of the next result can be found as a consequence of the 
so-called ``transfer method'' of \cite[Ch.~VI, Prop.~2.1]{Ma-Rockner92}
(see also \cite[Ch.~I, \S~9.2]{Bouleau-Hirsch91} in the case of 
a finite measure $\mm(X)<\infty$), applied
to the representation of $\ell$ through the $1$-excessive function
$u_\ell$ of \eqref{eq:88}.
%
%
\begin{proposition}
\label{prop:Daniell} 
 
Let us assume that $\cE$ is quasi-regular. 
Then for every $\ell\in \cVl +'$ 
there exists a 
(unique)  $\sigma$-finite and 
nonnegative Borel measure $\mu$ in $X$
 
such that 
every $\cE$-polar set 
is $\mu$-negligible and 
\begin{equation}
  \label{eq:13}
  \forall f\in \cV\quad
  \text{the $\cE$-q.c.~representative }\tilde f\in L^1(X,\mu),\quad
  \langle \ell,f\rangle=\int_X \tilde f\,\d\mu.
\end{equation}
If moreover
\begin{equation}
  \label{eq:27}
  \langle \ell,\varphi\rangle\le M\quad\text{for every }\varphi\in
  \cVl +,\quad \varphi\le 1\text{ $\mm$-a.e.\ in }X,
\end{equation}
then $\mu$ is a finite measure and $\mu(X)\le M$.
\end{proposition}
We will identify $\ell$ with $\mu$. Notice that if 
$\mu\in \cVl +'$ and $0\le \nu\le c\mu$, then also $\nu\in \cVl +'$
since
\begin{displaymath}
  \left|\int_X \tilde\varphi\,\d\nu\right|\le 
  \int_X |\tilde\varphi|\,\d\nu\le 
  c\int_X |\tilde\varphi|\,\d\mu
  \le c\|\mu\|_{\cV'}\,\|\varphi\|_\cV.
\end{displaymath}
The next Lemma provides a simple but important application
of the previous Proposition to the case of 
a function $u$ with measure-valued $\DeltaE u$.
We first recall a well known approximation procedure
(see e.g.\ \cite[Proof of Thm.~2.7]{Pazy83}), that 
will turn to be useful in the sequel.
  For $f\in L^2(X,\mm)$ let us set 
  \begin{equation}
    \label{eq:regularization}
    \begin{gathered}
      \mathfrak P_\eps f:=\frac1\eps\int_0^\infty \heat r
      f\,\kappa(r/\eps)\,
      \d r=
      \int_0^\infty \heat {\eps s}f\,\kappa(s)\,\d s,\quad\eps>0,\quad\text{where}\\
      \kappa\in \rmC^\infty_c(0,\infty)\text{ is a nonnegative kernel
        with }
      \int_0^\infty \kappa(r)\,\d r=1.
    \end{gathered}
  \end{equation}
  $\mathfrak P_\eps$ is positivity preserving and it is not difficult to check that for $\eps>0$ 
  $\mathfrak P_\eps f\in D(\DeltaE)$ and 
  for every $f\in L^p(X,\mm)$, $p\in [1,\infty]$, we have
  \begin{equation}
    \label{eq:26}
    \DeltaE f=-\frac1{\eps^2} \int_0^\infty \heat r
    f\,\kappa'(r/\eps)\,\d r
    \in L^p(X,\mm).
  \end{equation}
\begin{lemma}
  \label{le:Lapmeas}
  Let us assume that the strongly local Dirichlet form $\mathcal E$ 
  is  quasi-regular, according to Definition
  $\ref{def:quasi-regular}$. 
Let $u\in L^1\cap L^\infty(X,\mm)$ 
  be nonnegative and let $g\in L^1\cap L^2(X,\mm)$ such that 
  \begin{equation}
    \label{eq:10}
    \int_X u\DeltaE\varphi\,\d\mm\ge -\int_X g\varphi\,\d\mm\quad
    \text{for any nonnegative }\varphi\in
    D(\DeltaE)\cap L^\infty(X,\mm)
    \text{ with }\DeltaE\varphi\in L^\infty(X,\mm).
  \end{equation}
  Then 
  \begin{equation}
    \label{eq:1}
    u\in \cV,\quad
    \cE(u)\le \int_X u\,g\,\d\mm
    ,\quad
    \int_X g\,\d\mm\ge0,
  \end{equation}
  and there exists a unique finite Borel measure
  $\mu:=\mu_+-g\,\mm$ with $\mu_+\ge0$, $\mu_+(X)\le \int_X
  g\,\d\mm$
  such that every $\cE$-polar set is
  $|\mu|$-negligible, 
  the q.c.~representative of any function in $\cV$ belongs to
  $L^1(X,|\mu|),$ and
  \begin{equation}
    \label{eq:4}
    -\cE(u,\varphi)=-\int_X \Gbil u\varphi\,\d\mm=\int_X\tilde\varphi\,\d\mu
    \quad\forevery \varphi\in \cV.
  \end{equation}
\end{lemma}
\begin{proof}
  Let $u_\eps:=\mathfrak P_\eps u$, $\eps\ge0$, and notice that by the
  regularisation
  properties of $(\mathfrak P_\eps)_{\eps>0}$
  $u_\eps\in D(\DeltaE)$ with $\DeltaE u_\eps\in L^1\cap
  L^\infty(X,\mm)$.
  It follows that
  for every $\varphi\in L^2\cap L^\infty(X,\mm)$ nonnegative
  \begin{equation}
    \label{eq:7}
    \int_X \DeltaE u_\eps \varphi\,\d\mm=
    \int_X u\, \DeltaE \mathfrak P_\eps \varphi\,\d\mm \ge
    -\int_X g \mathfrak P_\eps\varphi\,\d\mm\ge 
    -\int_X g_+ \,\mathfrak P_\eps\varphi\,\d\mm,
  \end{equation}
  which in particular yields $\DeltaE u_\eps+\mathfrak P_\eps g\ge0$.
  Choosing $\varphi:=u_\eps$ in \eqref{eq:7} and inverting the sign of
  the inequality we obtain
  \begin{displaymath}
    \cE(u_\eps)=
    -\int_X \DeltaE u_\eps\,u_\eps\,\d\mm\le 
    \int_X u_\eps \, \mathfrak P_\eps g \,\d\mm
  \end{displaymath}
  We can then pass to the limit as $\eps\down0$ obtaining
  \eqref{eq:1}. 

  Moreover, taking nonnegative 
  functions $\phi,\psi\in L^2\cap L^\infty(X,\mm)$
  with $0\le \varphi(x)\le 1$ and 
  $\psi(x)>0$ for $\mm$-a.e.\ $x\in X$
  (such a function exists since $\mm$ is $\sigma$-finite)
  and setting $\varphi_n(x):=1\land(\varphi(x)+n\psi(x))$, \eqref{eq:7}
  applied
  to the differences $\varphi_{n+1}-\varphi_n\ge0$ (notice that
  $\varphi\equiv \varphi_0$),  yields that for
  every $n\ge 0 $
  \begin{displaymath}
    0\le \int_X (\DeltaE u_\eps+\mathfrak P_\eps g)\varphi\,\d\mm
    \le \int_X (\DeltaE u_\eps+\mathfrak P_\eps g)\varphi_n\,\d\mm
    \le \int_X (\DeltaE u_\eps+\mathfrak P_\eps g)\varphi_{n+1}\,\d\mm.
  \end{displaymath}
  Passing to the limit as $n\to\infty$, since $\varphi_n\uparrow 1$
  $\mm$-a.e.
  we obtain
  \begin{equation}
    \label{eq:28}
    0\le \int_X (\DeltaE u_\eps+\mathfrak P_\eps g)\varphi\,\d\mm
    \le  \int_X (\DeltaE u_\eps+\mathfrak P_\eps g)\,\d\mm=
    \int_X \mathfrak P_\eps g\,\d\mm=\int_X g\,\d\mm
  \end{equation}
  since $(\heat  t)_{t\ge0}$ is mass preserving and thus $\int_X \DeltaE u_\eps\,\d\mm=0$.
  Let us now denote by $\ell$ the linear functional in $\cV'$
  \begin{displaymath}
    \langle \ell,\varphi\rangle:=-\cE(u,\varphi)+\int_X g\,\varphi\,\d\mm
  \end{displaymath}
  Choosing a nonnegative $\varphi\in \cVz$ in \eqref{eq:7}
  and passing to the limit $\eps\downarrow0$ we easily
  find that $\ell\in \cVl +'$; if moreover $\varphi\le 1$ then
  \eqref{eq:28} yields
  \begin{displaymath}
    \langle \ell,\varphi\rangle=
    \lim_{\eps\down0}\Big(-\cE(u_\eps,\varphi)+\int_X \mathfrak P_\eps
    g\,\varphi\,\d\mm\Big)
    \le \int_X g\,\d\mm.
  \end{displaymath}
  Applying the previous Proposition \ref{prop:Daniell} we conclude.
\end{proof}
We denote by $\cMz$ the space of $u\in \cVz$ such that 
there exist $\mu=\mu_+-\mu_-$ with $\mu_\pm\in \cVl +'$ such that 
\begin{equation}
  \label{eq:30}
  -\cE(u,\varphi)=\int_X \tilde\varphi\,\d\mu\quad\forevery \varphi\in \cV,
  \qquad
  \text{and we will write $\DeltaEM u=\mu$.}
\end{equation}
For functions $u$ with measure-valued $\displaystyle\DeltaEM u$ we can
extend the calculus rule \eqref{eq:19}:

\begin{corollary}
  Under the same assumptions of 
  Proposition \ref{prop:Daniell}, for every $u\in \cMz$ and $f\in D(\DeltaE)\cap \cLz$
  we have
  $fu\in \cMz$ with
  \begin{equation}
    \label{eq:48}
    \DeltaEM(f\,u)=\tilde f\DeltaEM u+u\DeltaE f\,\mm+2\Gbil uf\mm.
  \end{equation}
\end{corollary}
\begin{proof}
  By \eqref{eq:60}
$\tilde f$ belongs to $L^\infty(X,|\mu|)$ where $\mu=\DeltaEM u$ 
and coincide with $f$ up to a $\mm$-negligible set; 
we have for every $\zeta\in \cVz$
\begin{align*}
  -\Bilinear{fu}\zeta&\topref{eq:18}=-
  \int \Big( f\Gbil u\zeta+u\Gbil f\zeta\Big)\,\d\mm \topref{eq:18}=
  -\int\Big( \Gbil u{f\zeta}+\Gbil f{u\zeta}-2\zeta\,
  \Gbil fu\Big)\,\d\mm\\&\topref{eq:30}=
  \int_X \tilde f\tilde\zeta \,\d\big(\DeltaEM u\big)+
  \int_X \big(u\DeltaE f+2\la \Gbil fu\big)\zeta\,\d\mm.
\end{align*}
By a standard approximation argument by truncation 
we extend the previous identity to arbitrary $\zeta\in \cV$
(notice that $\tilde f$ is essentially bounded and 
$\tilde\zeta \in L^1(X,|\mu|)$).
\end{proof}

\section{The Bakry-\'Emery condition and 
the measure-valued operator $\Gamma_2$}

\label{sec:BE}

\subsection{The Bakry-\'Emery condition}
Let us assume that the Dirichlet form $\cE$ 
admits a \emph{Carr\'e du champ} $\Gamma$ and let us 
introduce the multilinear form $\mathbf{\Gamma}_2$
\begin{equation}
\begin{aligned}
  \Gqlin{f,g}\varphi:=&
  \frac 12\int_X \Big(\Gbil fg\, {\DeltaE \varphi}-
  \big(\Gbil f{\DeltaE g}+\Gbil g{\DeltaE f}\big)\varphi\Big)\,\d\mm
  \quad
  (f,g,\varphi)\in D(\mathbf\Gamma_2),
\end{aligned}\label{eq:22}
\end{equation}
where $D(\mathbf{\Gamma}_2):=D_\cV(\DeltaE)\times D_\cV(\DeltaE) \times
  D_{L^\infty}(\DeltaE),$ and 
\begin{displaymath}
  D_\cV(\DeltaE)=\big\{f\in D(\DeltaE):
  \DeltaE f\in \cV\big\},\quad 
  D_{L^\infty}(\DeltaE):=\big\{\varphi\in D(\DeltaE)\cap L^\infty(X,\mm):\DeltaE\varphi\in L^\infty(X,\mm)\big\}.
\end{displaymath}
When $f=g$ we also set
\begin{displaymath}
  \Gqlin f\varphi:=\Gqlin{f,f}\varphi=
  \int_X \Big(\frac 12\Gq f\, {\DeltaE \varphi}-
  \Gbil f{\DeltaE f}\varphi\Big)\,\d\mm,
\end{displaymath}
so that 
\begin{displaymath}
  \Gqlin{f,g}\varphi=  \frac14\Gqlin{f+g}\varphi-\frac 14 \Gqlin{f-g}\varphi.
\end{displaymath}
$\mathbf\Gamma_2$ provides a weak version 
(inspired by \cite{Bakry06,Bakry-Ledoux06}) of 
the Bakry-\'Emery condition \cite{Bakry-Emery84,Bakry92}.
\begin{definition}[Bakry-\'Emery condition]
  \label{def:BE}
  We say that the strongly local Dirichlet form $\cE$ satisfies the $\BE K\infty$
  condition, $K\in \R$, if it admits a Carr\'e du Champ $\Gamma$ and
  \begin{equation}
  \label{eq:9}
  \Gqlin f\varphi\ge K\int_X \Gq f\,\varphi\,\d\mm\quad
  \text{for every }(f,\varphi)\in D(\mathbf{\Gamma}_2),\ \varphi\ge0.
  \tag{$\BE K\infty$}
\end{equation}
\end{definition}
\eqref{eq:9} is in fact equivalent \cite[Corollary 2.3]{AGS12} to the properties
\begin{equation}
  \label{eq:5}
  \Gq{\heat t f}\le \rme^{-2K t}\,\heat t{\Gq f}\quad \mm\text{-a.e.\
    in }X,\quad
  \forevery
  t\ge0,\ f\in \cV,
\end{equation}
and
\begin{equation}
  \label{eq:17}
  2\rmI_{2K}(t) \Gq{\heat t f}\le \heat t{f^2}-\big(\heat t f\big)^2
  \quad \mm\text{-a.e.\
    in }X,\quad
  \forevery
  t>0,\ f\in L^2(X,\mm),
\end{equation}
where $\rmI_{2K}(t)=\int_0^t \rme^{2K t}\,\d t$.

\subsection{An estimate for $\Gq f$ and multivariate calculus for $\Gamma_2$}
Let us introduce the space
\begin{equation}
  \label{eq:29}
  \DG:=\big\{f\in D(\DeltaE)\cap \cLz:\DeltaE f\in \cV\big\}.
\end{equation}
The following Lemma provides a further crucial regularity property for
$\Gq f$
when $f\in \DG$ and shows how to define a measure-valued
$\Edmeas f$ operator.
\begin{lemma}
	\label{le:Gcarre-in-V}
	Let $\cE$ be a strongly local and quasi-regular Dirichlet form.
        If \ref{eq:9} holds then for every
        $f\in \DG$
	we have $\Gq f\in \cMz$ with
	\begin{equation}
	    	\label{eq:G3}
      		\Dirichlet{\Gq f}\le -\int_X \Big(
                2K\Gq f^2+ \Gq f\Gbil f{\DeltaE f}\Big)\,\d\mm
	  \end{equation}
          and
          \begin{equation}
          \label{eq:14}
          \frac 12\DeltaEM\Gq f-\Gbil f{\DeltaE f}\mm\ge K\Gq f\mm.
        \end{equation}
        Moreover, $\cDv$ is an algebra (closed w.r.t.~pointwise
        multiplication)
        and
        if $\ff=(f_i)_{i=1}^n\in (\DG)^n$ then $\Phi(\ff)\in \DG$ 
        for every smooth function $\Phi:\R^n\to \R$ with $\Phi(0)=0$.
\end{lemma}
\begin{proof}
  Let us first notice that for every $f\in \cGz$ we have
  $\Gq f\in L^1(X,\mm)\cap L^\infty(X,\mm)\subset L^p(X,\mm)$ for
  every $p\in [1,\infty]$.
  
  If $f\in \DG$ 
  then \ref{eq:9} and Lemma \ref{le:Lapmeas} with
  $-g:=\Gbil f{\DeltaE f}+K\Gq f$ and $u:=\Gq f$
  yield $\Gq f\in \cV$. \eqref{eq:G3} then follows from \eqref{eq:1}.
  
  Since every function $\varphi\in \cVl +$ can be (strongly) approximated
  by nonnegative functions in $D_{L^\infty}(\DeltaE)$ 
  by means of the regularization operators \eqref{eq:regularization}, 
  \eqref{eq:14} is a direct consequence of \ref{eq:9}, 
  \eqref{eq:4}, \eqref{eq:30}, 

  We already observed in \lref{l6}, \S\ref{subsec:Leibnitz}, that if $f_1,f_2\in \DG$ 
  than $f_1 f_2\in D(\DeltaE)\cap \cGz$; \eqref{eq:19}  and \lref{l2} also show
  that $\DeltaE (f_1f_2)\in \cV$. A similar argument, based on
  \lref{l5},
  shows that $\Phi(\ff)\in \DG$ whenever $\ff\in (\DG)^n$.
\end{proof}
For every $f\in \DG$ we denote by $\Edmeas f$ 
the finite 
Borel measure 
\begin{equation}
  \label{eq:46}
  \Edmeas f:=\frac 12\DeltaEM \,\Gq f-
  \Gbil f{\DeltaE f}\mm.
\end{equation}
By Lemma \ref{le:Lapmeas}, $\Edmeas f$ has finite total variation,
since
\begin{equation}
  \label{eq:81}
  \Edmeas f=K\Gq f\mm+\mu_+,\quad  
  \text{with}\quad \mu_+\ge0,\quad
  \mu_+(X)\le -\int_X \Big(\Gbil f{\DeltaE f}+K\Gq f\Big)\,\d\mm.
\end{equation}
The measure $\Edmeas u$ vanishes on sets of $0$ capacity. We denote by
$\gdq u\in L^1(X,\mm)$ its density with respect to $\mm$: 
\begin{equation}
  \label{eq:47}
  \Edmeas f=\edmeas f\mm+\Esmeas f,\quad
  \Esmeas f\perp \mm,\quad
  \edmeas f \ge K\Gq f\quad\text{$\mm$-a.e.\ in $X$,\quad
    } \Esmeas f\ge0.
\end{equation}
The main point is that $\Edmeas\cdot$ can have a singular part
$\Esmeas \cdot$ w.r.t.\ $\mm$, but this is
nonnegative and it does not affect many crucial inequalities.

According to \eqref{eq:22} we also set for $f,g\in \DG$
\begin{equation}
  \label{eq:23}
  \Ebilmeas fg:=\frac14\Edmeas{f+g}-\frac 14\Edmeas{f-g} =
  \frac 12\Big(\DeltaEM \,\Gbil fg-
  \Gbil f{\DeltaE g}\mm-\Gbil g{\DeltaE f}\mm\Big),
\end{equation}
and similarly
\begin{equation}
  \label{eq:73}
  \ebilmeas fg:=\frac 14\gdq {f+g}-\frac 14\gdq {f-g},\quad
  \Ebilmeas fg=\ebilmeas fg\mm+\Ebilsmeas fg.
\end{equation}
The next lemma extends to the present nonsmooth setting 
the multivariate calculus for $\Gamma_2$ 
of \cite{Bakry85,Bakry92}.

\begin{lemma}[The fundamental identity]
  \label{le:fundamental-id}
  Under the same assumptions of the previous Lemma
  \ref{le:Gcarre-in-V}, let $\ff=(f^i)_{i=1}^n\in \DG^n$ and let $\Phi\in \rmC^3(\R^n)$
  with $\Phi(0)=0$.
  Then $\Phi(\ff)\in \DG$ and
  \begin{equation}
  \label{eq:50}
  \begin{aligned}
    &\Edmeas{\Phi(\ff)}
    =\sum_{i,j}\Phi_i(\tilde\ff)\,\Phi_j(\tilde\ff)\,
    \Ebilmeas{f^i}{f^j}\\
    &\qquad+\Big(2\sum_{i,j,k}\Phi_i(\ff)\Phi_{jk} (\ff)\tril{f^i}{f^j}{f^k}
    +\sum_{i,j,k,h}\Phi_{ik}(\ff)\Phi_{jh}(\ff)\Gbil{f^i}{f^j}\Gbil{f^k}{f^h}\Big)\mm,
  \end{aligned}
\end{equation}
where for $f,g,h\in \DG$
\begin{equation}
  \label{eq:25}
  \tril fgh=\frac 12\Big(\Gbil{g}{\Gbil fh}+\Gbil h{\Gbil fg}-\Gbil f{\Gbil gh}\Big).
\end{equation}
Similarly
  \begin{equation}
  \label{eq:50bis}
  \begin{aligned}
    &\edmeas{\Phi(\ff)}
    =\sum_{i,j}\Phi_i(\ff)\,\Phi_j(\ff)\,
    \ebilmeas{f^i}{f^j}\\
    &\qquad+2\sum_{i,j,k}\Phi_i(\ff)\Phi_{jk} (\ff)\tril{f^i}{f^j}{f^k}
    +\sum_{i,j,k,h}\Phi_{ik}(\ff)\Phi_{jh}(\ff)\Gbil{f^i}{f^j}\Gbil{f^k}{f^h}.
  \end{aligned}
\end{equation}
\end{lemma}
\begin{proof}
  The fact that $\Phi(\ff)\in \DG$ has been proved
  in Lemma \ref{le:Gcarre-in-V}.

  In the following we will assume that 
  the indices $i,j,h,k$ run from $1$ to $n$ and we will use
  Einstein summation convention.

  We set $g^{ij}:=\Gbil {f^i}{f^j}\in \cMz$,
  $\ell^i:=\DeltaE f^i\in \cV$,
$\phi_{i}:=\Phi_i(\tilde\ff),\,\phi_{ij}:=
\Phi_{ij}(\tilde\ff),\,\phi_{ijk}:=\Phi_{ijk}(\tilde\ff)$
in $\DG$; we will also consider the quasi-continuous representative.

By \eqref{eq:49} and Lemma \ref{le:Gcarre-in-V} we have
\begin{displaymath}
  \Gq{\Phi(\ff)}=g^{ij}\phi_i\phi_j\in \cMz
\end{displaymath}
Since $\phi_i\phi_j\in \DG$ by \lref{l6} and $g^{ij}\in \cMz$ by
Lemma \ref{le:Gcarre-in-V}, we can apply \eqref{eq:48} obtaining
\begin{align*}
  \frac 12 \DeltaEM\Gq{\Phi(\ff)}&=
  \frac12 \phi_i\phi_j\DeltaEM g^{ij}+\Big(\frac 12 g^{ij}\DeltaE(\phi_i\phi_j)+
  \Gbil{\phi_i\phi_j}{g^{ij}}\Big)\mm=I+\big(II+III\big)\mm.
\end{align*}
\begin{align*}
  II&\topref{eq:19}=\frac 12g^{ij}\Big(\phi_i\DeltaE\phi_j+\phi_j\DeltaE\phi_i+
  2\Gbil{\phi_i}{\phi_j}\Big)=
  g^{ij}\Big(\phi_i\DeltaE\phi_j+
  \Gbil{\phi_i}{\phi_j}\Big)
  \\&\topref{eq:15}=
  g^{ij}\Big[\phi_i\Big(\phi_{jk}\ell^k+\phi_{jkh}\,g^{kh}\Big)
  \topref{eq:49}+
  \phi_{ik}\,\phi_{jh}\,g^{kh}\Big]
  \intertext{where we used $g^{ij}=g^{ji}$,}
  III&\stackrel{(\ref{eq:49})}=\Big(\phi_{ik}\,\phi_j+\phi_{jk}\phi_i\Big)\Gbil{f^k}{g^{ij}}=
  \phi_i\phi_{jk}\Big(\Gbil{f^k}{g^{ij}}+\Gbil{f^j}{g^{ik}}\Big)
\end{align*}
where we used the identity $\phi_{ik}\,\phi_j
\Gbil{f^k}{g^{ij}}=\phi_i\phi_{jk}\Gbil{f^j}{g^{ik}}$ 
obtained by performing a cyclic permutation $i\to k\to j\to i$.

On the other hand
\begin{align*}
   \Gbil{\Phi(\ff)}{\DeltaE \Phi(\ff)}&\topref{eq:49}=
   \phi_i\Gbil{f^i}{\DeltaE \Phi(\ff)}
   \topref{eq:15}=\phi_i\Gbil{f^i}{\phi_k\ell^k+\phi_{kh}\,g^{kh}}
   \\&=
   \phi_i\phi_k\Gbil{f^i}{\ell^k}+
   \phi_i\ell^k\,\phi_{kj}\, g^{ij}+
   \phi_i\,g^{kh}\,\phi_{khj}\,g^{ij}+\phi_i\phi_{kh}\Gbil{f^i}{g^{kh}}
   \\&=
   \phi_i\phi_j\Gbil{f^i}{\ell^j}+
   \phi_i\ell^k\,\phi_{kj}\, g^{ij}+
   \phi_i\,g^{kh}\,\phi_{khj}\,g^{ij}+\phi_i\phi_{jk}\Gbil{f^i}{g^{jk}},
\end{align*}
where we changed $k$ with $j$ in the first term and $h$ with $j$ in
the last one. We end up with
  \begin{align*}
    \Edmeas{\Phi(\ff)}&=\frac12 \phi_i\phi_j\DeltaEM
    g_{ij}-\phi_i\phi_j\Gbil{f_i}{\ell_j}\mm \\&+
    \phi_{ik}\,\phi_{jh}\,g^{ij}\,g^{kh}\,\mm
    \\&+\phi_i\phi_{jk}\Big(\Gbil{f^k}{g^{ij}}+\Gbil{f^j}{g^{ik}}-\Gbil{f^i}{g^{jk}}\Big)\mm
  \end{align*}
  that gives \eqref{eq:50}.
\end{proof}
It could be useful to remember that in the smooth context
of a Riemannian manifold $(\M^n,\sfg)$ as 
for \eqref{eq:82} we have
\cite[Page 96]{Bakry92}
\begin{displaymath}
  \tril fgh=\langle \rmD^2 f\,\rmD g,\rmD h\rangle_\sfg.
\end{displaymath}
\subsection{A pointwise estimate for $\Gq{\Gq f}$}
Applying the previous results and adapting
the ideas of \cite{Bakry85} 
we can now state
our first fundamental estimates.
\begin{theorem}
  \label{thm:main1}
  Let $\cE$ be a strongly local and quasi-regular Dirichlet form.
  If \eqref{eq:9} holds then for every $f,g,h\in \DG$ 
  (so that $\Gq f,\Gq g,\Gq h\in \cVz$) we have
  (all the inequalities are to be intended
  $\mm$-a.e.\ in $X$)
  \begin{align}
    \label{eq:74}
    \Big|\tril fgh\Big|^2&\le \big(\edmeas f-K\Gq f\big)\Gq g\Gq h,\\    
    \label{eq:75}
    \sqrt{\Gq{\Gbil fg}}&\le \sqrt{\edmeas f-K\Gq f}\,\sqrt{\Gq g}+
    \sqrt{\edmeas g-K\Gq g}\,\sqrt{\Gq f},\\
  \label{eq:57}
  \Gq{\Gq f}&\le 4\Big(\gdq f-K\Gq f\Big)\Gq f.
\end{align}
\end{theorem}
\begin{proof}
Lemma \ref{le:Gcarre-in-V} shows that $\Gq f\in \cVz$.

We choose the polynomial $\Phi:\R^3\to \R$ defined by
\begin{equation}
  \label{eq:72}
  \Phi(\ff):=\lambda f^1+(f^2-a)(f^3-b)-ab,\quad \lambda ,a,b\in \R;
\end{equation}
keeping the same notation of 
Lemma \ref{le:fundamental-id} we have
\begin{gather*}
  \Phi_1(\ff)=\lambda,\quad
  \Phi_2(\ff)=f^3-b,\quad
  \Phi_3(\ff)=f^2-a\\
  \Phi_{23}(\ff)=\Phi_{32}(\ff)=1,\quad
  \Phi_{ij}(\ff)=0
  \quad\text{if } (i,j)\not\in\{(2,3),(3,2)\}.
\end{gather*}
If $\ff\in \DG$ Lemma \ref{le:Gcarre-in-V} yields $\Phi(\ff)\in \DG$
and we can then apply the inequality
\eqref{eq:47} obtaining
\begin{equation}
  \label{eq:53}
  \edmeas{\Phi(\ff)}\ge K\,\Gq{\Phi(\ff)}\quad
  \mm\text{-a.e.~in }X,
\end{equation}
where both sides of the inequality depend on $\lambda,a,b\in \R$.
Evaluating $\edmeas{\Phi(\ff)}$ by \eqref{eq:50bis},
and choosing a
countable dense set  $Q$ of the parameters $(\lambda,a,b)$ in $\R^3$,
for $\mm$-almost every $x\in X$
the previous inequality holds for
every $(\lambda,a,b)\in Q$.
Since the dependence of the left and right
side of the inequality w.r.t.\ $\lambda,a,b$ is continuous, we conclude
that for 
$\mm$-almost
every $x\in X$ 
the inequality holds for every $(\lambda,a,b)\in \R^3$. 
Apart from a $\mm$-negligible set, 
for every $x$ we can then choose $a:=f^2(x)$, $b=f^3(x)$
so that $\Phi_2(\ff)(x)=\Phi_3(\ff)(x)=0$  obtaining
\begin{align*}
  &\lambda^2\edmeas {f^1}+
  4\lambda \tril {f^1}{f^2}{f^3}
  +2\Big(\Gq {f^2}\Gq {f^3}+\Gbil {f^2}{f^3}^2\Big)\ge 
  K\lambda^2 \Gq {f^1}.
\end{align*}
Since $\lambda$ is arbitrary and 
\begin{displaymath}
  \Gq {f^2}\Gq {f^3}+\Gbil {f^2}{f^3}^2\le 2\Gq {f^2}\Gq{f^3},
\end{displaymath}
we eventually obtain
\begin{equation}
  \label{eq:54}
  \Big(\tril {f^1}{f^2}{f^3} \Big)^2\le \Big(\gdq {f^1}-K\Gq {f^1}\Big)\Gq {f^2}\Gq {f^3}
\end{equation}
that provides \eqref{eq:74}. 
\eqref{eq:75} then follows by first noticing that 
\begin{equation}
  \label{eq:76}
  \tril fgh+\tril gfh=
  \Gbil{\Gbil fg}h, 
\end{equation}
so that 
\begin{equation}
  \label{eq:77}
  \Big|\Gbil{\Gbil fg}h\Big|
  \le \left[\sqrt{\edmeas f-K\Gq f}\sqrt{\Gq g}+
    \sqrt{\edmeas g-K\Gq g}\sqrt{\Gq f}
    \right]\sqrt{\Gq h}.
\end{equation}
We argue now by approximation, fixing $f,g\in \DG$ and approximating an
arbitrary $h\in \cVz$ with a sequence $h_n\in \DG$ 
(e.g.\ by \eqref{eq:regularization}) converging to
$h$ in energy with 
\begin{displaymath}
  \Gq {h_n}\to \Gq{h},\quad
  \Gbil {h_n}{\Gbil fg}\to \Gbil h{\Gbil fg}
\end{displaymath}
pointwise and in
$L^1(X,\mm)$, thanks to \eqref{eq:5} (see also
Remark 2.5 and (4.5) of \cite{AGS12}): \eqref{eq:77} thus hold for arbitrary $h\in \cVz$ and
we can then choose $h:=\Gbil fg$ obtaining \eqref{eq:75}.
\eqref{eq:57} then follows by choosing $g:=f$ in \eqref{eq:75}.
\end{proof}
\begin{corollary}
  \label{cor:kuwada}
  Under the same assumption of
  Theorem \ref{thm:main1}, 
  for every  $f\in \cV$ and $\alpha\in [1/2,1]$ we have
  \begin{equation}
    \label{eq:41}
    \Gq {\LHeat tf}^{\alpha}\le \rme^{-2\alpha Kt} \LHeat t\big(\Gq f^{\alpha}\big).
  \end{equation}
\end{corollary}
\begin{proof}
  We adapt here the strategy of \cite{Wang11}. 
  Since the case $\alpha=1$ has been already covered by \eqref{eq:5},
  we can also assume $1/2\le \alpha<1$.

  We consider the concave and smooth function
  $\eta_\eps(r):=\big(\eps+r\big)^\alpha-\eps^\alpha$,
  $\eps>0$, $r\ge
  0$, and for a time $t>0$, a nonnegative $\zeta\in \cVz$, and
  an arbitrary $f\in \DG$ 
  we define the curves
  \begin{equation}
    \label{eq:58}
    f_{\tau}:=\heat{\tau} f,\ \zeta_s:=\heat s\zeta,\ 
    u_{\tau}:=\Gq{f_{\tau}},\quad
    G_\eps(s):=\int_X \eta_\eps(u_{t-s})\,\zeta_s\,\d\mm,\quad \tau,s\in
    [0,t].
  \end{equation}
  Notice that $\eta_\eps$ is smooth and Lipschitz; a direct
  computation yields
  \begin{equation}
    \label{eq:59}
    \eta_\eps(r)\le r^\alpha,\quad
    r\,\eta_\eps'(r)\ge \alpha\eta_\eps,
    \quad
    2\eta_\eps'(r)+4r\,\eta_\eps''(r)\ge
    0.
  \end{equation}
  Moreover, for every $s\in [0,t]$ $f_{t-s}\in \DG$ so that 
  $u_{t-s}\in \cV\cap L^1\cap L^\infty(X,\mm),$
  \begin{displaymath}
    \frac \d{\d s} u_{t-s}=
    -2\Gbil {f_{t-s}}{\DeltaE f_{t-s}},\quad
    \frac \d{\d s} \eta_\eps(u_{t-s})=
    -2\eta_\eps'(u_{t-s})\Gbil {f_{t-s}}{\DeltaE f_{t-s}}\quad
    \text{in }L^1\cap L^2(X,\mm).
  \end{displaymath}
  Differentiating with respect to $s\in (0,t)$ we get
  \renewcommand{\LHeat}[2]{#2_{#1}}
  \begin{align*}
    G'(s)&=\int_X \Big(\eta_\eps(u_{t-s})\,\DeltaE
    \zeta_s -2\eta_\eps'(u_{t-s})\Gbil {f_{t-s} }{\DeltaE f_{t-s}}\zeta_s
    \Big)\,\d\mm
    \\&=
    -\int_X
    \eta_\eps'(u_{t-s})\Gbil{\Gq{f_{t-s}}}{\zeta_s}
    +2\Gbil{f_{t-s}}{\DeltaE
      f_{t-s}}\eta_\eps'(u_{t-s})\zeta_s
    \Big)\,\d\mm
  \\&=
    -\int_X\Big(
    \Gbil{\Gq{f_{t-s}}}{\eta_\eps'(u_{t-s})\zeta_s}
    -\Gq {\Gq {\LHeat{t-s}
        f}}\eta_\eps''(u_{t-s})\zeta_s
    +2\Gbil{f_{t-s}}{\DeltaE f_{t-s}}\eta_\eps'(u_{t-s})\zeta_s
    \Big)\,\d\mm
    \\&=
    2\int_X \eta_\eps'(\tilde u_{t-s})\tilde\zeta_s\,\d\Edmeas{f_{t-s}}+
    \int \Gq {\Gq {\LHeat{t-s}
        f}}\eta_\eps''(u_{t-s})\zeta_s\,\d\mm
        \\&\ge
        2\int_X \eta_\eps'(u_{t-s})\zeta_s\,\gamma_2(f_{t-s})\,\d\mm
        +
        4\int_X \eta_\eps''(u_{t-s}) \big(\gamma_2(f_{t-s})-Ku_{t-s}\big)u_{t-s}\,\zeta_s\,\d\mm
    \\&=
    \int_X
    \Big(2\eta_\eps'(u_{t-s})+4\eta_\eps''(u_{t-s})u_{t-s}\Big)\Big(\gamma_2(f_{t-s})-K
    u_{t-s}\Big)\zeta_s\,\d\mm
    +2K\int_X \eta_\eps'(u_{t-s})u_{t-s}\zeta_s\,\d\mm
    \\&\ge
    2K\int_X \eta_\eps'(u_{t-s})u_{t-s}\zeta_s\,\d\mm
    \topref{eq:59}\ge 2\alpha\,K\int_X \eta_\eps(u_{t-s})\zeta_s\,\d\mm=
    2\alpha\, KG_\eps(s).
  \end{align*}
  thanks to \eqref{eq:57}.

Since $G$ is continuous, we obtain $G_\eps(0)\rme^{2\alpha Kt}\le G_\eps(t)$
which yields, after passing to the limit as $\eps\down0$
\begin{equation}
  \label{eq:16}
  \rme^{2\alpha Kt}\int_X {\Gq{\heat t f}}^\alpha\,\zeta\,\d\mm\le 
  \int_X {\Gq f}^\alpha\,\heat t \zeta\,\d\mm=
  \int_X \heat t\big({\Gq f}^\alpha\big)\, \zeta\,\d\mm.
\end{equation}
Since $\DG$ is dense in $\cV$ we can extend
\eqref{eq:16} to arbitrary $f\in \cV$ and then obtain 
\eqref{eq:41}, since $\zeta$ is arbitrary.
\end{proof}

\section{$\RCD K\infty$-metric measure spaces}

\label{sec:RCD}

In this section we will apply the previous result to 
prove new contraction properties w.r.t.~transport costs
(in particular $W_p$ Wasserstein distance) for the heat flow in 
$\RCD K\infty$ metric measure spaces.

\subsection{Basic notions}

\label{subsec:basic}
\subsubsection*{Metric measure spaces, transport and Wasserstein
  distances, entropy}

\newcommand{\Probabilitiesp}[1]{\mathscr P_p(#1)}
We will quickly recall a few basic facts concerning
optimal transport of probability measures, also to fix notation;
we refer to
\cite{AGS08,Villani09} for more details.

 Let $(X,\sfd)$ be a complete and separable metric space endowed with 
a Borel measure $\mm$ 
satisfying 
\begin{equation}
  \label{eq:2}
  \supp(\mm)=X,\qquad
  \mm(B_r(\bar x))\le \sfc_1 \exp(\sfc_2 r^2)\quad\forevery r>0,
  \tag{$\mm$-exp}
\end{equation}
for some constants $\sfc_1,\sfc_2\ge0$ and a point $\bar x\in X$.

Recall that for every Borel probability measure $\mu\in \Probabilities
Y$ in a separable metric space $Y$ and every Borel map $\rr:Y\to X$,
the push-forward $\rr_\sharp\mu\in \Probabilities X$ is defined by
$\rr_\sharp\mu(B)=\mu(\rr^{-1}(B))$ for every $B\in \BorelSets X$.
If $\mu_i\in \Probabilities X$, $i=1,2$, we denote by 
$\AdmissiblePlanII {\mu_1}{\mu_2}$ the collection of 
all couplings $\mmu$ 
between $\mu_1$ and $\mu_2$, i.e.\ measures in $\Probabilities{X\times
  X}$ 
whose marginals $\pi^i_\sharp \mmu$ coincide with $\mu_i$
(here $\pi^i(x_1,x_2)=x_i$).
Given a nondecreasing continuous function $h:[0,\infty)\to
[0,\infty)$, we consider the transport cost 
\begin{equation}
  \label{eq:21}
  \calC_h(\mu_1,\mu_2):=\min\Big\{\int_{X\times X}
  h(\sfd(x,y))\,\d\mmu(x,y):
  \mmu\in \AdmissiblePlanII{\mu_1}{\mu_2}\Big\},
\end{equation}
where we implicitly assume that the minimum is $+\infty$ if 
couplings with finite cost do not exist.
In the particular case $h(r):=r^p$ we set
\begin{equation}
  \label{eq:32}
  W_p(\mu_1,\mu_2):=\big(\calC_h(\mu_1,\mu_2)\big)^{1/p},\quad h(r):=r^p,
\end{equation}
and we also set
\begin{equation}
  \label{eq:33}
  W_\infty(\mu_1,\mu_2)=\min\Big\{\|\sfd\|_{L^\infty(X\times
    X,\smmu)}:\mmu\in \AdmissiblePlanII{\mu_1}{\mu_2}\Big\}
  =\lim_{p\up\infty}W_p(\mu_1,\mu_2).
\end{equation}
Denoting by $\Probabilitiesp X$ the space of Borel probability
measures
with finite $p$-th moment, i.e.
\begin{equation}
  \label{eq:34}
  \mu\in \Probabilitiesp X\quad\Longleftrightarrow\quad
  \int_X \sfd^p(x,\bar x)\,\d\mu(x)<\infty\quad\text{for some (and
    thus any) $\bar x\in X$},
\end{equation}
$(\Probabilitiesp X,W_p)$ is a complete and separable metric space.

The relative entropy of a measure $\mu\in \ProbabilitiesTwo X$ is defined
as 
\begin{equation}
  \label{eq:35}
  \ent \mu:=
  \begin{cases}
    \int_X \rho\log \rho\,\d\mm&\text{if }\mu=\rho\mm\ll \mm,\\
    +\infty&\text{otherwise.}
  \end{cases}
\end{equation}
The entropy functional is well defined
and lower semicontinuous w.r.t.\ $W_2$ convergence
(see e.g.\ \cite[\S 7.1]{AGS11a}

\subsubsection*{The Cheeger energy and its $L^2$-gradient flow}

We first recall that the metric slope of a 
Lipschitz function $f:X\to\R$ is defined by
\begin{equation}
  \label{eq:36}
  |\rmD f|(x):=\limsup_{y\to x}\frac{|f(y)-f(x)|}{\sfd(x,y)}.
\end{equation}
The Cheeger energy \cite{Cheeger00,AGS11a} is obtained as the $L^2$-lower semicontinuous
envelope 
of the functional $f\mapsto \frac 12\int_X |\rmD f|^2\,\d\mm$:
\begin{equation}
  \label{eq:37}
  \C (f):=\inf\Big\{\liminf_{n\to\infty}\frac12\int_X |\rmD
  f_n|^2\,\d\mm:f_n\in \Lip_b(X),\quad
  f_n\to f\text{ in }L^2(X,\mm)\Big\}.
\end{equation}
If $\C(f)<\infty$ it is possibile to show that the collection 
\begin{displaymath}
  S(f):=\Big\{G\in L^2(X,\mm):\exists f_n\in \Lip_b(X),\ f_n\to f,\ |\rmD
  f_n|\weakto G
  \text{
    in }L^2(X,\mm)\Big\}
\end{displaymath}
admits a unique element of minimal norm, 
\emph{the minimal weak upper gradient}  $|\rmD f|_w$, 
that it is also minimal with
respect to the order structure \cite[\S 4]{AGS11a}, i.e.\ 
\begin{equation}
  \label{eq:84}
  |\rmD f|_w\in S(f),\quad |\rmD f|_w\le G\quad \text{$\mm$-a.e.}\quad\forevery 
  G\in S(f).
\end{equation}
By $|\rmD f|_w$ we can also represent $\C(f)$ as
\begin{equation}
  \label{eq:51}
  \C(f)=\frac 12\int_X |\rmD f|_w^2\,\d\mm.
\end{equation}
It turns out that $\C$ is a $2$-homogeneous, l.s.c., convex functional in
$L^2(X,\mm)$, whose proper domain $D(\C):=\{f\in
L^2(X,\mm):\C(f)<\infty\}$ is a dense linear subspace of $L^2(X,\mm)$.

Its $L^2$-gradient flow 
is a continuous semigroup of contractions $(\sfh_t)_{t\ge0}$ in
$L^2(X,\mm)$,
whose continuous trajectories $f_t=\sfh_t f$, $t\ge0$ and $f\in
L^2(X,\mm)$, 
are locally Lipschitz curves in $(0,\infty)$ with values in
$L^2(X,\mm)$ characterised by the differential inclusion
\begin{equation}
  \label{eq:31}
  \frac \d{\d t}f_t+\partial\C(f_t)\ni0\quad\text{a.e.\ in }(0,\infty).
\end{equation}

\subsection{$\RCD K\infty$-spaces}

In order to state the main equivalent definitions of $\RCD K\infty$
spaces,
let us first recall a list of relevant properties:
\begin{description}
\item[\QCh:] The Cheeger energy is quadratic, i.e.\ 
  \begin{equation}
  \label{eq:38}
  \C(f+g)+\C(f-g)=2\C(f)+2\C(g)\quad\text{for every }f,g\in D(\C).
\end{equation}
\item[\CDKI:] 
  The entropy functional is displacement $K$-convex in 
    $\ProbabilitiesTwo X$ \cite{Sturm06I,Lott-Villani09}, i.e.\ 
    for every $\mu_0,\mu_1\in D(\entv)\subset \ProbabilitiesTwo X$ and $t\in [0,1]$ 
    there exists $\mu_t\in \ProbabilitiesTwo X$ such that
    \begin{equation}
      \label{eq:61}
      \begin{gathered}
        W_2(\mu_0,\mu_t)=tW_2(\mu_0,\mu_1),\quad
        W_2(\mu_t,\mu_1)=(1-t)W_2(\mu_0,\mu_1),\\
        \ent {\mu_t}\le (1-t)\ent{\mu_0}+
        t\ent{\mu_1}-\frac K2 t(1-t)W_2(\mu_0,\mu_1).
      \end{gathered}
    \end{equation}
  \item[\Length:] $(X,\sfd)$ is a length space, i.e.\
    for every $x_0,x_1\in X$ and $\eps>0$ there exists an
    $\eps$-middle point $x_\eps\in X$ such that 
    \begin{equation}
      \label{eq:63}
      \sfd(x_0,x_\eps)<\frac 12\sfd(x_0,x_1)+\eps,\quad
      \sfd(x_1,x_\eps) <\frac 12\sfd(x_0,x_1)+\eps.
    \end{equation}
    \item[\Cont:] Every bounded function $f\in D(\C)$ with $|\rmD f|_w\le 1$
      admits a continuous representative.
    \item[\Wcont:] For every $f_0,f_1\in
      L^2(X,\mm)$ with
      $f_i\mm\in \ProbabilitiesTwo X$ we have
      \begin{equation}
        \label{eq:64}
        W_2(\sfh_t f_0\,\mm,\sfh_t f_1\,\mm)\le \rme^{-K
          t}W_2(f_0\mm,f_1\mm)\quad t\ge0.
      \end{equation}
    \item[\BEKI:] 
      Assuming that the Cheeger energy is quadratic, then the
      Dirichlet form $\cE:=2\C$ satisfies
      the Bakry-\'Emery condition according to Definition
      \ref{def:BE}.
    \item[\EVIK:] 
      For every $\bar\mu\in \ProbabilitiesTwo X$ there
      exists a curve $(\mu_t)_{t\ge0}\subset D(\entv)$ such that 
      $\lim_{t\down0}\mu_t=\bar\mu$ and 
      \begin{equation}
        \label{eq:65}
        \frac \d{\d t_+}\frac 12 W_2^2(\mu_t,\nu)+\frac K2
        W_2^2(\mu_t,\nu)\le \ent \nu-\ent{\mu_t}\quad t>0.
      \end{equation}
\end{description}
Let us now recall the main equivalence results:
\begin{theorem}
  Let $(X,\sfd,\mm)$ be a complete, length, and separable metric measure space 
  satisfying condition \eqref{eq:2} and let $K\in \R$.
  The following set of conditions for $(X,\sfd,\mm)$ are equivalent:
  \begin{enumerate}[\rm (I)]
  \item \QCh\ and \CDKI;
  \item \QCh, \Cont, and \Wcont;
  \item
    \QCh, \Cont, and \BEKI;
  \item \EVIK.
  \end{enumerate}
  Moreover, if one of the above conditions hold then 
  $\cE:=2\C$ is a strongly local and 
   quasi-regular
  (see Definition \ref{def:quasi-regular})
  Dirichlet form,
  it 
  admits the Carr\'e du Champ
  \begin{equation}
    \label{eq:52}
    \Gq f=|\rmD f|_w^2\quad\forevery f\in D(\C),
  \end{equation}
  the subdifferential $\partial\C$ is single-valued and coincides with
  the linear generator $\DeltaE$, 
  $(\sfh_t)_{t\ge0}=(\heat t)_{t\ge0}$, and for every
  $\bar\mu=f\mm\in \ProbabilitiesTwo X$ with $f\in L^2(X,\mm)$ 
  the curve $\mu_t=\sfh_t f\mm$ is the unique solution of
  \eqref{eq:65}. 
  Eventually, any essentially bounded function $f\in D(\C)$ 
  with $|\rmD f|_w\le L$ admits a $L$-Lipschitz representative $\tilde
  f$,
  and for every $f\in \Lip_b(X),g\in \rmC_b(X)$ we have
  \begin{equation}
    \label{eq:85}
    |\rmD f|_w\le |\rmD f|;\quad
    |\rmD f|_w\le g\quad \Longrightarrow\quad
    |\rmD f|\le g.
  \end{equation}
\end{theorem}
\begin{proof}
  The implication (I)$\Leftrightarrow$(IV) has been proved in
  \cite[Thm.~5.1]{AGS11b} in the case when $\mm \in\ProbabilitiesTwo
  X$ and extended to the general case by \cite{AGMR12}.
  (IV)$\Rightarrow$(II),(III) has been proved in \cite[Thm.~6.2,
  Thm.~6.10]{AGS11b} and the relations
  (II)$\Leftrightarrow$(III)$\Rightarrow$(IV) have been proved in
  \cite[Thm.~3.17, Cor.~3.18, Cor.~4.18]{AGS12}.  \eqref{eq:52}
  follows from \cite[Thm.~4.18]{AGS11b}; see \cite[Prop.~3.11]{AGS12}
  for \eqref{eq:85}.

   Let us eventually check that $\cE$ is quasi-regular, by applying Lemma
  \ref{le:qr-criterium}. 
  $\langle\mathrm{QR}.2'\rangle$ is always true for a Cheeger
  energy, since Lipschitz functions are dense in $\cV$ by
  \eqref{eq:37}.
  
  When $\mm(X)<\infty$ we can always choose $X_n:=X$ 
  and $\langle\mathrm{QR}.1'\rangle$ 
  reduces to the tightness property \eqref{eq:tight}, that has
  been proved in \cite[Lemma 6.7]{AGS11b}, following an argument of
  \cite[Proposition IV.4.2]{Ma-Rockner92}.
  In the general case we can adapt the same argument: we 
  recall here the various steps for the easy of the reader.
  
  Let us fix a point $\bar x\in X$ and let us set
  $X_n:=B_n(\bar x)$. 
  In order to prove that $(\overline X_n)_{n\in \N}$ is an $\cE$-nest,
  we introduce the $1$-Lipschitz cut-off functions 
  $\psi_n:X\to [0,1]$
  $$\psi_n(x):=0\lor (n-\sfd(x,\bar x))\land 1,\quad
  \text{so that}\quad
  \psi_n(x)=
  \begin{cases}
    1&\text{if }x\in \overline X_{n-1},\\
    0&\text{if }x\in X\setminus X_n,
  \end{cases}
  \quad
  \psi_n(x)\uparrow 1\text{ as }n\to\infty.
  $$
  For every $f\in \cV$ we can consider the approximations
  $f_n:=\psi_n f$ in $\cVrestr{\bar X_n}$.
  The Lebesgue's Dominated Convergence Theorem shows that 
  $f_n\to f$ strongly in $L^2(X,\mm)$ as $n\to\infty$.
  The Leibnitz
  rule yields
  \begin{displaymath}
    |\rmD (f-f_n)|_w\le \big(|\rmD f|_w+|f|\big)\nchi_{X\setminus B_{n-1}(\bar x)}    
  \end{displaymath}
  so that $\lim_{n\to\infty}\cE(f-f_n)= 0$ as well.
  This shows that $f_n\to f$ strongly in $\cV$ and 
  $(\overline X_n)_{m\in \N}$ is an $\cE$-nest.
  
  In order to prove $\langle\mathrm{QR}.1'\rangle$, we fix $n\in \N$,
  we consider a dense sequence 
  $(x_j)_{j\in \N} $ in $X_{n+1}$, and we define 
  the functions $w_k:X\to [0,1]$
  $$
  w_k(x):=\psi_{n+1}(x)\land \min_{1\leq j\leq k}\sfd (x,x_j)\quad
  x\in X.
  $$
  It is easy to check that $w_k$ are $1$-Lipschitz and 
  pointwise nonincreasing, they satisfy $0\leq w_k\leq \psi_{n+1}\le 1$ and
  the pointwise limit $w_k\downarrow 0$ as $k\to\infty$, so that
  $w_k\to w$ strongly in $L^2(X,\mm)$ since $\supp(w_k)\subset \overline
  X_{n+1}$ and
  $\mm(\overline X_{n+1})<\infty$.
  The finiteness of $\mm(\overline X_{n+1})$ also
  yields that $(w_k)_{k\in \N}$ is bounded in 
  $\cV$, so that $w_k\weakto 0$ weakly in $\cV$ as $k\to\infty$.

  The Banach-Saks theorem ensures the existence of an increasing subsequence
  $(k_h)_{h\in \N}$ such that the Cesaro means
  $v_h:=\frac1h \sum_{i=1}^h w_{k_i}
  $
  converge to $0$ strongly in $\cV$. This implies \cite[Thm.~1.3.3]{Chen-Fukushima12}
that a subsequence $(v_{h(l)})$ of $(v_h)$ converges to $0$
quasi-uniformly, 
i.e. for all integers $m\geq 1$
there exists a closed set $G_m\subset X$ such that $\cp(X_{n+1}\setminus G_m)<1/m$
and $v_{h(l)}\to 0$ uniformly on $G_m$. 
As $w_{k_{h(l)}}\leq v_{h(l)}$, if we set $F_m=\cup_{i\leq m}G_i$, we
have that $w_{k_{h(l)}}\to 0$ as $l\to\infty$ 
uniformly on $F_m$ for all $m$ and
$\cp(X_{n+1}\setminus F_m)\leq 1/m$. 

Therefore, for every $\delta>0$ we can find an integer 
$p\in \N$ such that
$w_p<\delta$ on $F_m$; 
since $\psi_{n+1}(x)\equiv 1$ when $x\in \overline
X_n$, the definition of $w_p$ implies
$$
\forall\, x\in \overline X_n\cap F_m\quad\exists 
j\in \N,\ j\le p:\ \sfd(x,x_j)<\delta,\quad
\text{i.e.}\quad
\overline X_n\cap F_m\subset\bigcup_{j=1}^p B(x_i,\delta).
$$
Since $\delta$ is arbitrary this proves that $K_{n,m}:=
\overline X_n\cap F_m$ is totally bounded, hence compact and
$\cp( X_n\setminus K_{n,m})\le \cp(X_{n+1}\setminus F_m)\le 1/m$.
\end{proof}
\begin{definition}
  We say that 
  $(X,\sfd,\mm)$ is $\RCD K\infty$-metric measure space  
  if it is complete, separable and length, $\mm$ satisfies
  \eqref{eq:2}, and at least one of the 
  (equivalent) properties \textup{(I)--(IV)} holds.
\end{definition}
By Corollary \ref{cor:kuwada} we thus obtain:
\begin{corollary}
  \label{cor:dual-improved}
  If $(X,\sfd,\mm)$ is a $\RCD K\infty$ metric measure space, then
  for every $t>0$, $\beta\in [1,2]$ and $f\in \cVz$
  \begin{equation}
    \label{eq:83}
    |\rmD \heat t f|^\beta\le \rme^{-\beta Kt}\heat t(|\rmD f|_w^\beta).
  \end{equation} 
\end{corollary}
We call $\sfH_t \bar \mu$ the unique solution to \eqref{eq:65}:
by \cite[Prop.~3.2]{AGS12} $(\sfH)_{t\ge0}$ can be extended
in a unique way to a semigroup of weakly continuous operators
in $\Probabilities X$ satisfying
\begin{equation}
  \label{eq:66}
  \lim_{t\down0} \sfH_t\mu=\mu \text{ in }\Probabilities X,\quad
  W_2(\sfH_t\mu_0,\sfH_t \mu_1)\le \rme^{-K t}W_2(\mu_0,\mu_1)\
  \forevery \mu_0,\mu_1\in \Probabilities X.
\end{equation}
In particular we can consider the fundamental solutions 
\begin{equation}
  \label{eq:62}
  \varrho_{t,x}:=\sfH_t \delta_x\in \ProbabilitiesTwo X.
\end{equation}
\subsection{New contraction properties for the heat flow $(\sfH_t)_{t\ge0}$}

Let us fix a parameter $K\in \R$ and 
for every  nondecreasing cost function
$h:[0,\infty)\to[0,\infty)$ 
let us consider the perturbed cost functions
\newcommand{\perth}[2]{h_{#1#2}}
\begin{equation}
  \label{eq:67}
  \perth  Kt (r):=h(\rme^{Kt}\,r),
\end{equation}
and the associated transportation costs $\calC_h,\ \calC_{\perth Kt}$.

\begin{theorem}
  \label{thm:main}
  Let $(X,\sfd,\mm)$ be a $\RCD K\infty$ metric measure space.
  Then 
  \begin{enumerate}[\rm i)]
  \item For every $x,y\in X$ 
    the fundamental solutions $\varrho_{t,x}$, $\varrho_{t,y}$ defined
    by 
    \eqref{eq:62} satisfy
    \begin{equation}
      \label{eq:68}
      W_\infty(\varrho_{t,x},\varrho_{t,y})\le \rme^{-Kt}\sfd(x,y).
    \end{equation}
  \item For every $\mu,\nu\in \Probabilities X$ 
    \begin{equation}
      \label{eq:69}
      \calC_{\perth Kt}(\sfH_t \mu,\sfH_t\nu)\le \calC_{h}(\mu,\nu).
    \end{equation}
  \item For every $\mu,\nu\in \Probabilities X$ and every $p\in
    [1,\infty]$
    \begin{equation}
      \label{eq:70}
      W_p(\sfH_t\mu,\sfH_t\nu)\le \rme^{-Kt}W_p(\mu,\nu).
    \end{equation}
  \end{enumerate}
\end{theorem}
\begin{proof}
  \textbf{i)} follows from \eqref{eq:83} by the Kuwada's duality argument 
  \cite[Prop.~3.7]{Kuwada10}
  as developed by \cite[Lemma 3.4, Theorem 3.5]{AGS12}. 

  \textbf{ii)} 
  Let $\mu,\nu$ with $\calC_h(\mu,\nu)<\infty$ and 
  let $\ggamma\in \AdmissiblePlanII \mu\nu$ be an optimal plan for $\calC_h$.
  We may use a measurable selection theorem (see for instance
  \cite[Theorem 6.9.2]{Bogachev07}
  to select in a $\ggamma$-measurable way optimal plans
  $\ggamma_{x,y}$ for $W_\infty$ between $\varrho_{t,x}$ and
  $\varrho_{t,y}$. Then, we define
  \begin{displaymath}
    \ssigma:=\int_{X\times X}\ggamma_{x,y}\,\d\ggamma(x,y).
  \end{displaymath}
  Notice that $\ssigma\in \Gamma(\sfH_t\mu,\sfH_t\nu)$ since e.g.
  for every $\varphi\in \rmC_b(X)$ we have
  \begin{align*}
    \int_{X\times X}\varphi(x)\,\d\ssigma(x,y)&=
    \int_{X\times X}\int_{X\times
      X}\varphi(x)\,\d\gamma_{u,v}(x,y)\,\d\ggamma(u,v)=
    \int_{X\times X}\int_X
    \varphi(x)\,\d\varrho_{t,u}(x)\,\d\ggamma(u,v)
    \\&=\int_{X}\int_X
    \varphi(x)\,\d\varrho_{t,u}(x)\,\d\mu(u)=
    \int_X \varphi(x)\,\d\sfH_t\mu(x),
  \end{align*}
  and a similar computation holds integrating functions depending only
  on $y$.
  Therefore, since \eqref{eq:68} yields
  \begin{equation}
    \label{eq:71}
    \sfd(x,y)\le \rme^{-Kt}\sfd(u,v)\quad 
    \text{for $\ggamma_{u,v}$-a.e.\ $(x,y)\in X\times X$},
  \end{equation}
  \begin{align*}
    \calC_{\perth Kt}(\sfH_t\mu,\sfH_t\nu)&\le 
    \int_{X\times X}\perth Kt\big(\sfd(x,y)\big)\,\d\ssigma(x,y)
    \\&=
    \int_{X\times X}\int_{X\times
      X}h\big(\rme^{Kt}\sfd(x,y)\big)\,\d\gamma_{u,v}(x,y)\,\d\ggamma(u,v)
    \\&\topref{eq:71}\le 
    \int_{X\times X}\int_{X\times
      X}h\big(\sfd(u,v)\big)\,\d\gamma_{u,v}(x,y)\,\d\ggamma(u,v)
    =\int_{X\times X}
    h\big(\sfd(u,v)\big)\,\d\ggamma(u,v)
    \\&=\calC_h(\mu,\nu).
  \end{align*}
  \textbf{iii)} follows immediately by \eqref{eq:69} by choosing 
  $h(r):=r^p$ so that $\perth Kt(r)=\rme^{pKt}r^p$.
\end{proof}

\def\cprime{$'$}

\end{document}